\documentclass[12pt]{article}
\usepackage{amsfonts}
\usepackage{}
\usepackage{wasysym,amssymb,eufrak,indentfirst,graphicx,cite,amsthm,color}
\usepackage[bookmarksnumbered, colorlinks, plainpages]{hyperref}
\newtheorem{Theorem}{Theorem}[section]
\theoremstyle{definition}
\newtheorem{Corollary}[Theorem]{Corollary}
\newtheorem{Lemma}[Theorem]{Lemma}
\newtheorem{Proposition}[Theorem]{Proposition}
\newtheorem{Question}[Theorem]{Question}
\newtheorem{Definition}[Theorem]{Definition}
\newtheorem{Example}[Theorem]{Example}
\theoremstyle{definition}

\begin{document}
\baselineskip 17pt

\title{On weakly $\sigma$-permutable subgroups of finite groups\thanks{Research was supported by the NNSF  of China (11371335) and Wu Wen-Tsun Key Laboratory of Mathematics of Chinese Academy of Sciences.}}

\author{Chi Zhang, Zhenfeng Wu, W. Guo\thanks{Corresponding author}\\
{\small  Department of Mathematics, University of Science and Technology of China}\\
{\small Hefei, 230026, P.R. China}\\
{\small E-mail:zcqxj32@mail.ustc.edu.cn, zhfwu@mail.ustc.edu.cn, wbguo@ustc.edu.cn}\\ \\}

\date{}
\maketitle

\begin{abstract}
 Let $G$ be a finite group
 and $\sigma=\{\sigma_{i}|i\in I\}$ be a partition of the set of all primes $\mathbb{P}$.
 A set $\mathcal {H}$ of subgroups of $G$ with $1 \in \mathcal {H}$ is said to be a complete Hall $\sigma$-set of $G$ if every non-identity member of $\mathcal {H}$ is a Hall $\sigma_{i}$-subgroup of $G$.
 A subgroup $H$ of $G$ is said to be $\sigma$-permutable if $G$ possesses a complete Hall $\sigma$-set $\mathcal {H}$
such that $HA^{x}=A^{x}H$ for all $A\in \mathcal {H}$ and all $x\in G$.
 We say that a subgroup $H$ of $G$ is weakly $\sigma$-permutable in $G$
 if there exists a $\sigma$-subnormal subgroup $T$ of $G$
 such that $G=HT$ and $H\cap T\leq H_{\sigma G}$,
 where $H_{\sigma G}$ is the subgroup of $H$ generated by all those subgroups of $H$ which are $\sigma$-permutable in $G$.
 By using this new notion,
 we establish some new criterias for a group $G$ to be a $\sigma$-soluble and supersoluble, and also we give the conditions under which a normal subgroup of $G$ is  hypercyclically embedded.
\end{abstract}

\let\thefootnoteorig\thefootnote
\renewcommand{\thefootnote}{\empty}

\footnotetext{Keywords: Finite group; $\sigma$-permutable subgroup; weakly $\sigma$-permutable subgroup; $\sigma$-soluble; supersoluble; hypercyclically embedded}

\footnotetext{Mathematics Subject Classification (2010): 20D10, 20D15, 20D20} \let\thefootnote\thefootnoteorig

\section{Introduction}
Throughout this paper, all groups  are finite and $G$ always denotes a  group.
Moreover, $n$ is an integer, $\mathbb{P}$ is the set of all primes.
The symbol $\pi(n)$ denotes the set of all primes dividing $n$ and $\pi(G)=\pi(|G|)$,
the set of all primes dividing the order of $G$.
$|G|_{p}$ denotes the order of the Sylow $p$-subgroup of $G$.

In what follows, $\sigma=\{\sigma_{i}|i\in I\}$ is some partition of $\mathbb{P}$,
that is, $\mathbb{P}=\bigcup_{i\in I}\sigma_{i}$ and $\sigma_{i}\cap \sigma_{j}=\emptyset$ for all $i\neq j$.
$\Pi$ is always supposed to be a non-empty subset of $\sigma$ and $\Pi^{'}=\sigma\backslash \Pi$.
We write $\sigma(n)=\{\sigma_{i}|\sigma_{i}\cap \pi(n)\neq\emptyset\}$ and $\sigma(G)=\sigma(|G|)$.

Following \cite{AN1, WS1}, $G$ is said to be $\sigma$-primary  if $G = 1$ or $|\sigma(G)| = 1$;
$n$ is said to be a $\Pi$-number if $\pi(n)\subseteq\bigcup_{\sigma_{i}\in \Pi}\sigma_{i}$;
a subgroup $H$ of $G$ is called a $\Pi$-subgroup of $G$ if $|H|$ is a $\Pi$-number;
a subgroup $H$ of $G$ is called a Hall $\Pi$-subgroup of $G$ if $H$ is a $\Pi$-subgroup of $G$ and $|G:H|$ is a $\Pi^{'}$-number.
A set $\mathcal {H}$ of subgroups of $G$ with $1 \in \mathcal {H}$ is said to be a complete Hall $\Pi$-set of $G$ if every non-identity member of $\mathcal {H}$ is a Hall $\sigma_{i}$-subgroup of $G$ for some $\sigma_{i}\in \Pi$ and $\mathcal {H}$ contains exact one Hall $\sigma_{i}$-subgroup of $G$ for every $\sigma_{i}\in \Pi\cap \sigma(G)$.
In particular,
when $\Pi=\sigma$,
we call $\mathcal {H}$ a complete Hall $\sigma$-set of $G$.
$G$ is said to be {\sl $\Pi$-full} if $G$ possesses a complete Hall $\Pi$-set;
{\sl a $\Pi$-full group of Sylow type} if every subgroup of $G$ is a $D_{\sigma_{i}}$-group
for all $\sigma_{i}\in \Pi\cap \sigma(G)$.
In particular, $G$ is said to be {\sl $\sigma$-full} (or {\sl $\sigma$-group}) if $G$ possesses a complete Hall $\sigma$-set;
{\sl a $\sigma$-full group of Sylow type} if every subgroup of $G$ is a $D_{\sigma_{i}}$-group for all $\sigma_{i}\in \sigma(G)$.
A subgroup $H$ of $G$ is called \cite{AN1}  $\sigma$-subnormal in $G$ if there is a subgroup chain
$H=H_0\leq H_1\leq\cdots \leq H_t=G$ such that either $H_{i-1}$ is normal in $H_{i}$
or $H_{i}/(H_{i-1})_{H_{i}}$ is $\sigma$-primary for all $i=1,2,\cdots,t$.

In the past 20 years, a large number of researches have involved finding and applying some generalized complemented subgroups.
For example, a subgroup $H$ of $G$ is said to be $c$-normal \cite{WY1} in $G$
if $G$ has a normal subgroup $T$ of $G$
such that $G=HT$ and $H\cap T\leq H_{G}$,
where $H_{G}$ is the normal core of $H$.
A subgroup $H$ of $G$ is said to be weakly $s$-permutable \cite{AN3} in $G$
if $G$ has a subnormal subgroup $T$
such that $G=HT$ and $H\cap T\leq H_{sG}$,
where $H_{sG}$ is the largest $s$-permutable subgroup of $G$ contained in $H$
(note that a subgroup $H$ of $G$ is said to be $s$-permutable in $G$
if $HP=PH$ for any Sylow subgroup $P$ of $G$).
A subgroup $H$ of $G$ is said to be $\sigma$-permutable \cite{AN1} in $G$ if $G$ possesses a complete Hall $\sigma$-set $\mathcal {H}$
such that $HA^{x}=A^{x}H$ for all $A\in \mathcal {H}$ and all $x\in G$.
By using the above special supplemented subgroups and other generalized complemented subgroups,
people have obtained a series of interesting results (see \cite{AN1, WY1, AN3, AM, AB3, AB2, W, WS4, BL, WH} and so on).
Now, we consider the following new generalized supplemented subgroup:

\begin{Definition}
A subgroup $H$ of $G$ is said to be weakly $\sigma$-permutable in $G$
if there exists a $\sigma$-subnormal subgroup $T$ of $G$
such that $G=HT$ and $H\cap T\leq H_{\sigma G}$,
where $H_{\sigma G}$ is the subgroup of $H$ generated by all those subgroups of $H$ which are $\sigma$-permutable in $G$.
\end{Definition}

Following \cite{AN2}, $H_{\sigma G}$ is called $\sigma$-core of $H$.

It is clear that every $c$-normal subgroup,  every $s$-permutable subgroup, every weakly $s$-permutable subgroup and every $\sigma$-permutable subgroup of $G$ are all weakly $\sigma$-permutable in $G$.
However, the following example shows that the converse is not true.

\begin{Example}\label{e1}
 Let $G=(C_{7}\rtimes C_{3})\times A_{5}$,
where $C_{7}\rtimes C_{3}$ is a non-abelian group of order 21
and $A_{5}$ is the alternating group of degree 5.
Let $B$ be a subgroup of $A_{5}$ of order 12
and $A$ be a Sylow $5$-subgroup of $A_{5}$.
Let $\sigma=\{\sigma_{1},\sigma_{2}\}$,
where $\sigma_{1}=\{2,3,5\}$ and $\sigma_{2}=\{2,3,5\}^{'}$.
Then $B$ is weakly $\sigma$-permutable in $G$.
In fact, let $T=(C_{7}\rtimes C_{3})\times A$,
then $C_{7}\rtimes C_{3}\leq T_{G}$
and $|G:C_{7}\rtimes C_{3}|=5\cdot 2^2\cdot 3$ is a $\sigma_{1}$-number.
Hence $G/T_{G}$ is a $\sigma_{1}$-group,
and so $T$ is $\sigma$-subnormal in $G$.
Since $T\cap B=1$ and $G=BT$,
which means that $B$ is weakly $\sigma$-permutable in $G$.
But $B$ is neither weakly $s$-permutable in $G$ nor $c$-normal in $G$.
In fact, if there exists a subnormal subgroup $K$ of $G$
such that $G=BK$ and $B\cap K\leq B_{sG}$.
Then $B_{sG}$ is subnormal in $G$ by \cite[Lemma 2.8]{AN3},
and so is subnormal in $A_{5}$ by \cite[A, (14.1)]{Doerk} .
It follows that $B_{sG}=1$ for $A_{5}$ is a simple group.
Hence $|G:K|=|B|=2^2\cdot 3$.
But since $1<A_{5}<A_{5}C_{7}<G$ is a chief series of $G$ and also a composition series of $G$, $G$ has no  subnormal subgroup $K$ whose index is $2^2\cdot 3$ by Jordan-H$\ddot{o}$lder theorem.
Therefore $B$ is not weakly $s$-permutable in $G$. Consequently, $B$ is neither $s$-permutable nor $c$-normal in $G$.

Now let $H=BC_{3}$.
Then $H$ is weakly $\sigma$-permutable in $G$ but not $\sigma$-permutable in $G$,
Indeed, let $T=C_{7}A_{5}$.
Then $G=HT$,
$T$ is normal in $G$ and $H\cap T=B$.
It is easy to see that $\mathcal {H}=\{A_{5}C_{3}, C_{7}\}$ is a complete Hall $\sigma$-set of $G$.
Since $H_{\sigma G}$ is $\sigma$-subnormal in $G$ by Lemma \ref{permutable} (4) below and {\cite[Theorem B]{AN1}},
$H_{\sigma G}\leq O_{\sigma_{1}}(G)$ by Lemma \ref{subnormal}(8) below.
Clearly, $O_{\sigma_{1}}(G)\leq C_{G}(O_{\sigma_{2}}(G))=C_{G}(C_{7})=C_{7}A_{5}$.
Hence $H_{\sigma G}\leq C_{7}A_{5}$.
But since $B(A_{5}C_{3})^{x}=BA_{5}C_{3}^{x}=A_{5}C_{3}^{x}=C_{3}^{x}A_{5}=(A_{5}C_{3})^{x}B$ for all $x\in G$,
$B$ is $\sigma$-permutable in $G$ for $C_{7}\unlhd G$.
Hence $B\leq H_{\sigma G}\leq C_{7}A_{5}$,
which implies that $B=H_{\sigma G}$.
So $H$ is weakly $\sigma$-permutable in $G$,
But $H$ is not $\sigma$-permutable in $G$ for $H_{\sigma G}=B<H$.
 \end{Example}

Following \cite{AN1}, $G$ is called:
(i) $\sigma$-soluble if every chief factor of $G$ is $\sigma$-primary;
(ii) $\sigma$-nilpotent if $H/K\rtimes (G/C_{G}(H/K))$ is $\sigma$-primary
for every chief factor $H/K$ of $G$.

The result in \cite{AN1, AN3, WS2, SS, WY1} are the motivation for the following:

\begin{Question}
Let $G$ be a $\sigma$-full group of Sylow type.
What is the structure of $G$ provided that some subgroups are weakly $\sigma$-permutable in $G$?
\end{Question}

In this paper, we obtain the following results.

\begin{Theorem}\label{Th1}
Let $G$ be a $\sigma$-full group of Sylow type
and every Hall $\sigma_{i}$-subgroup of $G$ is weakly $\sigma$-permutable in $G$ for every $\sigma_{i}\in \sigma(G)$.
Then $G$ is $\sigma$-soluble.
\end{Theorem}

\begin{Theorem}\label{Th3}
Let $G$ be a $\sigma$-full group of Sylow type
and $\mathcal {H}=\{1,W_{1}, W_{2}, \cdots, W_{t}\}$ be a complete Hall $\sigma$-set of $G$
such that $W_{i}$ is a nilpotent $\sigma_{i}$-subgroup for all $i=1,\cdots,t$.
Suppose that the maximal subgroups of any non-cyclic $W_{i}$ is weakly $\sigma$-permutable in $G$.
Then $G$ is supersoluble.
\end{Theorem}

The following results now follow immediately from Theorems \ref{Th1} and \ref{Th3}.

\begin{Corollary}
If every sylow subgroup is weakly $s$-permutable in $G$,
then $G$ is soluble.
\end{Corollary}

\begin{Corollary} \label{cor3}\textup{(See Huppert \cite[Chap. VI, Theorem 10.3]{HU})}
If every Sylow subgroup of $G$ is cyclic, then  $G$ is supersoluble .
\end{Corollary}

\begin{Corollary}\textup{(See Miao \cite[Corollary 3.4]{LM})}
If all maximal subgroups of every Sylow subgroup of $G$ are weakly $s$-permutable in $G$,
then $G$ is supersoluble.
\end{Corollary}

\begin{Corollary}\textup{(See Skiba \cite[ Theorem 1.4]{AN3})}
If all maximal subgroups of every non-cyclic Sylow subgroup of $G$ are weakly $s$-permutable in $G$,
then $G$ is supersoluble.
\end{Corollary}

\begin{Corollary} \textup{(See Srinivasan \cite[Theorem 1]{SS})}
If all maximal subgroups of every Sylow subgroup of $G$ are normal in $G$,
then $G$ is supersoluble.
\end{Corollary}

\begin{Corollary} \textup{(See Srinivasan \cite[Theorem 2]{SS})}
If all maximal subgroups of every Sylow subgroup of $G$ are $s$-permutable in $G$,
then $G$ is supersoluble.
\end{Corollary}

\begin{Corollary} \textup{(See Wang \cite[Theorem 4.1]{WY1})}
If all maximal subgroups of every Sylow subgroup of $G$ are $c$-normal in $G$,
then $G$ is supersoluble.
\end{Corollary}

Recall that a normal subgroup $E$ of $G$ is called hypercyclically embedded in $G$ and is denoted by $E\leq Z_{\mathfrak{U}}(G)$ (see \cite[p. 217]{SR}) if either $E=1$ or $E\neq1$ and every chief factor of $G$ below $E$ is cyclic, where the symbol $Z_{\mathfrak{U}}(G)$ is the $\mathfrak{U}$-hypercentre of $G$, that is, the product of all normal hypercyclically embedded subgroups of $G$.
Hypercyclically embedded subgroups play an important role in the theory of  groups (see \cite{AB2, W, SR, MW} ) and the conditions under  which a normal subgroup is  hypercyclically embedded in $G$ were found by many authors (see the books \cite{AB2, W, SR, MW} and the recent papers \cite{BL, WS2, AN4, AN5, LS, WST} ).

On the base of Theorem \ref{Th3}, we will prove the following result.

\begin{Theorem}\label{Th4}
Let $G$ be a $\sigma$-full group of Sylow type
and $\mathcal {H}=\{1,W_{1}, W_{2}, \cdots, W_{t}\}$ be a complete Hall $\sigma$-set of $G$
such that $W_{i}$ is nilpotent for all $i=1,\cdots,t$.
Let $E$ be a normal subgroup of $G$.
If every maximal subgroup of $W_{i}\cap E$ is weakly $\sigma$-permutable in $G$ for all $i=1,2,\cdots, t$,
then $E\leq Z_{\mathfrak{U}}(G)$.
\end{Theorem}

The following results directly follow from Theorem \ref{Th4}.

\begin{Corollary} \label{cor2}
Let $\mathfrak{F}$ be a saturated formation containing all supersoluble groups
and let $E$ be a normal subgroup of $G$ with $G/E\in \mathfrak{F}$.
Suppose that $G$ is a $\sigma$-full group of Sylow type
and $\mathcal {H}=\{1,W_{1}, W_{2}, \cdots, W_{t}\}$ is a complete Hall $\sigma$-set of $G$
such that $W_{i}$ is nilpotent for all $i=1,\cdots,t$.
If every maximal subgroup of $W_{i}\cap E$ is weakly $\sigma$-permutable in $G$,
for all $i=1,2,\cdots, t$,
then $G\in \mathfrak{F}$.
\end{Corollary}

\begin{Corollary} \textup{(See Asaad \cite [Theorem 4.1]{AM2})}
Let $G$ be a group with a normal subgroup $E$
such that $G/E$ is supersoluble.
If every maximal subgroup of every Sylow subgroups of $E$ is $s$-permutable in $G$,
then $G$ is supersoluble.
\end{Corollary}

\begin{Corollary} \textup{(See Asaad \cite [Theorem 1.3] {AM3})}
Let $\mathfrak{F}$ be a saturated formation containing all supersoluble groups
and let $E$ be a normal subgroup of $G$ with $G/E\in \mathfrak{F}$.
If the maximal subgroups of every Sylow subgroup of E are $s$-permutable in $G$,
then $G\in \mathfrak{F}$.
\end{Corollary}

\begin{Corollary}  \textup{(See Wei \cite [Corollary 1] {WH})}
Let $\mathfrak{F}$ be a saturated formation containing all supersoluble groups
and let $E$ be a normal subgroup of $G$ with $G/E\in \mathfrak{F}$.
If the maximal subgroups of every Sylow subgroup of E are $c$-normal in $G$,
then $G\in \mathfrak{F}$.
\end{Corollary}

All unexplained terminologies and notations are standard,
as in \cite{W} and \cite{Doerk}.
\section{Preliminaries}
We use $\mathfrak{S}_{\sigma}$ and $\mathfrak{N}_{\sigma}$ to denote the classes of all $\sigma$-soluble groups and $\sigma$-nilpotent groups.

\begin{Lemma}\label{soluble}\textup(See {\cite[Lemma 2.1]{AN1}}
The class $\mathfrak{S}_{\sigma}$ is closed under taking direct products,
homomorphic images and subgroups.
Moreover,
any extension of the $\sigma$-soluble group by a $\sigma$-soluble group is a $\sigma$-soluble group as well.
\end{Lemma}

Following \cite{AN1} and \cite{WS1}, $O^{\Pi}(G)$ to denote the subgroup of $G$ generated by all its $\Pi^{'}$-subgroups.
Instead of $O^{\{\sigma_{i}\}}(G)$ we write $O^{\sigma_{i}}(G)$.
$O_{\Pi}(G)$ to denote the subgroup of $G$ generated by all its normal $\Pi$-subgroups.

\begin{Lemma}\label{subnormal}\textup(See {\cite[Lemma 2.6]{AN1}} and {\cite[Lemma 2.1]{WS1}})
Let $A,K$ and $N$ be subgroups of $G$. Suppose that $A$ is $\sigma$-subnormal in $G$ and $N$ is normal in $G$.

$(1)$ $A\cap K$ is $\sigma$-subnormal in $K$.

$(2)$ If $K$ is a $\sigma$-subnormal subgroup of $A$,
then $K$ is $\sigma$-subnormal in $G$.

$(3)$ If $K$ is $\sigma$-subnormal in $G$,
then $A\cap K$ and $\langle A,K\rangle$ are $\sigma$-subnormal in $G$.

$(4)$ $AN/N$ is $\sigma$-subnormal in $G/N$.

$(5)$ If $N\leq K$ and $K/N$ is $\sigma$-subnormal in $G/N$,
then $K$ is $\sigma$-subnormal in $G$.

$(6)$ If $K\leq A$ and $A$ is $\sigma$-nilpotent,
then $K$ is $\sigma$-subnormal in $G$.

$(7)$ If $|G:A|$ is a $\Pi$-number,
then $O^{\Pi}(A)=O^{\Pi}(G)$.

$(8)$ If $G$ is $\Pi$-full and $A$ is a $\Pi$-group,
then $A\leq O_{\Pi}(G)$.

\end{Lemma}

Let $\mathcal {L}$ be some non-empty set of subgroups of $G$
and $E$ a subgroup of $G$.
Then a subgroup $A$ of $G$ is called $\mathcal {L}$-permutable
if $AH=HA$ for all $H\in\mathcal {L}$;
$\mathcal {L}^{E}$-permutable
if $AH^{x}=H^{x}A$ for all $H\in\mathcal {L}$ and all $x\in E$.
In particular,
a subgroup $H$ of $G$ is $\sigma$-permutable
if $G$ possesses a complete Hall $\sigma$-set $\mathcal {H}$
such that $H$ is $\mathcal{H}^{G}$-permutable.

\begin{Lemma}\label{permutable}\textup(See {\cite[Lemma 2.8]{AN1}} and {\cite[Lemma 2.2]{WS1}})
Let $H,K$ and $N$ be subgroups of $G$.
Let $\mathcal {H}=\{H_{1},H_{2},\cdots ,H_{t}\}$ be a complete Hall $\sigma$-set of $G$
and $\mathcal {L}=\mathcal {H}^{K}$.
Suppose that $H$ is $\mathcal {L}$-permutable and $N$ is normal in $G$.

$(1)$ If $H\leq E\leq G$,
then $H$ is $\mathcal {L}^{\ast}$-permutable,
where $\mathcal {L}^{\ast}=\{H_{1}\cap E,H_{2}\cap E,\cdots ,H_{t}\cap E\}^{K\cap E}$.
In particular,
if $G$ is a $\sigma$-full group of Sylow type and $H$ is $\sigma$-permutable in $G$,
then $H$ is $\sigma$-permutable in $E$.

$(2)$ The subgroup $HN/N$ is $\mathcal {L}^{\ast\ast}$-permutable,
where $\mathcal {L}^{\ast\ast}=\{H_{1}N/N,\cdots ,H_{t}N/N\}^{KN/N}$.

$(3)$ If $G$ is a $\sigma$-full group of Sylow type and $E/N$ is a $\sigma$-permutable subgroup of $G/N$,
then $E$ is $\sigma$-permutable in $G$.

$(4)$ If $K$ is $\mathcal {L}$-permutable,
then $\langle H,K\rangle$ is $\mathcal {L}$-permutable \cite[A, Lemma 1.6(a)]{Doerk}.
In particular,
$H_{\sigma G}$ is $\sigma$-permutable in $G$.
Moreover, if $G$ is a $\sigma$-full group of Sylow type, then $H_{\sigma G}$ is a $\sigma$-subnormal of $G$(see {\cite[Theorems B and C]{AN1}}).

\end{Lemma}

\begin{Lemma}\label{normalizes}\textup(See {\cite[Lemma 3.1]{AN1}})
Let $H$ be a $\sigma_{1}$-subgroup of a $\sigma$-full group $G$.
Then $H$ is $\sigma$-permutable in $G$
if and only if $O^{\sigma_{1}}(G)\leq N_{G}(H)$.
\end{Lemma}

\begin{Lemma}\label{main}
Let $G$ be a $\sigma$-full group of Sylow type and $H\leq K\leq G$.

$(1)$ If $H$ is weakly $\sigma$-permutable in $G$,
then $H$ is weakly $\sigma$-permutable in $K$.

$(2)$ Suppose that $N$ is a normal subgroup of $G$ and $N\leq H$.
Then $H/N$ is weakly $\sigma$-permutable in $G/N$ if and only if $H$ is weakly $\sigma$-permutable in $G$.

$(3)$ If $N$ is a normal subgroup of $G$,
then for every weakly $\sigma$-permutable subgroup $H$ of $G$ with $(|H|,|N|)=1$,
$HN/N$ is weakly $\sigma$-permutable in $G/N$.
\end{Lemma}

\begin{proof}
$(1)$ Suppose that there exists a $\sigma$-subnormal subgroup $T$ of $G$
such that $G=HT$ and $H\cap T\leq H_{\sigma G}$.
Since $H\leq K$,
$K=H(K\cap T)$.
By Lemma \ref{subnormal} (1), $K\cap T$ is $\sigma$-subnormal in $K$.
Moreover, $H\cap(K\cap T)=H\cap T\leq H_{\sigma G}\leq H_{\sigma K}$ by Lemma \ref{permutable} (1)(4).
Hence, $H$ is weakly $\sigma$-permutable in $K$.

$(2)$ First assume that there exists a $\sigma$-subnormal subgroup $T$ of $G$
such that $G=HT$ and $H\cap T\leq H_{\sigma G}$.
Then $G/N=(H/N)(TN/N)$, $TN/N$ is $\sigma$-subnormal in $G/N$ by Lemma \ref{subnormal} (4)
and $H/N\cap TN/N=(H\cap T)N/N\leq H_{\sigma G}N/N\leq(H/N)_{\sigma(G/N)}$ by Lemma \ref{permutable} (2).
This shows that $H/N$ is weakly $\sigma$-permutable in $G/N$.

Conversely, assume that $H/N$ is weakly $\sigma$-permutable in $G/N$.
Then $G/N=(H/N)(T/N)$ and $H/N\cap T/N\leq(H/N)_{\sigma(G/N)}$,
where $T/N$ is $\sigma$-subnormal in $G/N$.
So $G=HT$ and $T$ is $\sigma$-subnormal in $G$ by Lemma \ref{subnormal} (5).
Let $(H/N)_{\sigma(G/N)}=E/N$.
Then $E$ is $\sigma$-permutable in $G$ by Lemma \ref{permutable} (3)(4).
Hence $H\cap T\leq E\leq H_{\sigma G}$.
This shows that $H$ is weakly $\sigma$-permutable in $G$.

$(3)$ Assume that there exists a $\sigma$-subnormal subgroup $T$ of $G$
such that $G=HT$ and $H\cap T\leq H_{\sigma G}$.
Then $G/N=(HN/N)(TN/N)$.
Since $(|H|,|N|)=1$,
$(|HT\cap N:H\cap N|,|HT\cap N:T\cap N|)=(|(HT\cap N)H:H|,|(HT\cap N)T:T|)=1$.
Hence $N=N\cap HT=(N\cap H)(N\cap T)=N\cap T$ by \cite[A, Lemma 1.6]{Doerk}.
It follows that $N\leq T$.
Hence $(HN/N)\cap (TN/N)=(H\cap T)N/N\leq H_{\sigma G}N/N\leq(HN/N)_{\sigma(G/N)}$
by Lemma \ref{permutable} (2)(4).
Besides, by Lemma \ref{subnormal} (4), $T/N$ is $\sigma$-subnormal in $G/N$.
Thus $HN/N$ is weakly $\sigma$-permutable in $G/N$.
\end{proof}

\begin{Lemma}\label{XC}\textup{(See \cite [Lemma 2.12]{CGA}})
Let $P$ be a normal $p$-subgroup of a group $G$ .
If $P/\Phi(P)\leq Z_{\mathcal {U}}(G/\Phi(P))$,
then $P\leq Z_{\mathcal {U}}(G)$.
\end{Lemma}

 \section{Proof of Theorem \ref{Th1}}

{\bf Proof of Theorem \ref{Th1}.}
Assume that this is false and let $G$ be a counterexample of minimal order.
Then  $|\sigma(G)|>1$.

$(1)$ $G/N$ is $\sigma$-soluble,
for every non-identity normal subgroup $N$ of $G$.

Let $N$ be a non-identity normal subgroup of $G$ and $H/N$ is any Hall $\sigma_{i}$-subgroup of $G/N$,
where $\sigma_{i}\cap \pi(G/N)\neq\emptyset$.
Then $H/N = H_{i}N/N$ for some Hall $\sigma_{i}$-subgroup $H_{i}$ of $G$.
By the hypothesis,
there exists a $\sigma$-subnormal subgroup $T$ of $G$
such that $G=H_{i}T$ and $H_{i}\cap T\leq (H_{i})_{\sigma G}$.
Then $G/N=(H_{i}N/N)(TN/N)=(H/N)(TN/N)$.
Since $|H_{i}N\cap T:H_{i}\cap T|=|(H_{i}N\cap T)H_{i}:H_{i}|$ is a $\sigma_{i}'$-number,
$(|H_{i}N\cap T:H_{i}\cap T|,|H_{i}N\cap T:N\cap T|)=1$.
Hence $H_{i}N\cap T=(H_{i}\cap T)(N\cap T)$ by \cite[A, Lemma 1.6]{Doerk}.
Consequently, $(H_{i}N/N)\cap (TN/N)=(H_{i}N\cap T)N/N=(H_{i}\cap T)N/N\leq (H_{i})_{\sigma G}N/N\leq(H_{i}N/N)_{\sigma(G/N)}$
by Lemma \ref{permutable} (2)(4).
By Lemma \ref{subnormal} (4), $TN/N$ is $\sigma$-subnormal in $G/N$.
Therefore $H/N$ is weakly $\sigma$-permutable in $G/N$.
This shows that $G/N$ satisfies the hypothesis.
The minimal choice of $G$ implies that $G/N$ is $\sigma$-soluble.

$(2)$ $G$ is not a simple group.

Suppose that $G$ is a non-abelian simple group.
Then $1$ is the only proper $\sigma$-subnormal subgroup of $G$.
Let $H_i$ be a non-identity Hall $\sigma_{i}$-subgroup of $G$,
where $\sigma_{i}\in \sigma(G)$.
By the hypothesis and $|\sigma(G)|>1$,
we have $G=H_{i}G$ and $H_{i}=H_{i}\cap G\leq(H_{i})_{\sigma G}$.
By Lemma \ref{permutable} (4), $(H_{i})_{\sigma G}$ is $\sigma$-subnormal in $G$,
so $H_{i}=(H_{i})_{\sigma G}=1$, a contradiction.
So we have (2).

$(3)$ Let $R$ be a minimal normal subgroup of $G$,
then $R$ is $\sigma$-soluble.

Let $H$ be any  Hall $\sigma_{i}$-subgroup of $R$,
where $\sigma_{i}\cap \pi(R)\neq\emptyset$.
Then there exists a Hall $\sigma_{i}$-subgroup $H_{i}$ of $G$ such that $H=H_{i}\cap R$.
By the hypothesis,
there exists a $\sigma$-subnormal subgroup $T$ of $G$
such that $G=H_{i}T$ and $H_{i}\cap T\leq(H_{i})_{\sigma G}$.
Since $|H_{i}T\cap R:H_{i}\cap R|=|(H_{i}T\cap R)H_{i}:H_{i}|$ is a $\sigma_{i}^{'}$-number,
$(|H_{i}T\cap R:H_{i}\cap R|,|H_{i}T\cap R:T\cap R|)=1$.
Hence, $R=R\cap H_{i}T=(H_{i}\cap R)(R\cap T)=H(R\cap T)$ by \cite[A, Lemma 1.6(c)]{Doerk}.
Since $R\cap T$ is $\sigma$-subnormal in $R$ by Lemma \ref{subnormal} (1)
and $H\cap R\cap T=(R\cap H_{i})\cap (R\cap T)\leq(H_{i})_{\sigma G}\cap R\leq (H)_{\sigma R}$ by Lemma \ref{permutable} (1)(4), $R$ satisfies the hypothesis.
The minimal choice of $G$ implies that $R$ is $\sigma$-soluble.

$(4)$ Final contradiction.

In veiw of (1), (2) and (3), we have that $G$ is $\sigma$-soluble by Lemma \ref{soluble}.
The final contradiction completes the proof of the theorem.

 \section{Proof of Theorem \ref{Th3}}

First we prove the following proposition, which is a main step of the proof of Theorem \ref{Th3}.

\begin{Proposition}\label{prop1}
Let $G$ be a $\sigma$-full group of Sylow type
and $\mathcal {H}=\{1,W_{1}, W_{2}, \cdots, W_{t}\}$ be a complete Hall $\sigma$-set of $G$
such that $W_{i}$ is a nilpotent $\sigma_{i}$-subgroup for all $i=1,\cdots,t$,
and let the smallest prime $p$ of $\pi(G)$ belongs to $\sigma_{1}$.
If every maximal subgroup of $W_{1}$ is weakly $\sigma$-permutable in $G$,
then $G$ is soluble.
\end{Proposition}

\begin{proof}
First note that if $G$ is $\sigma$-soluble, then every chief factor $H/K$ of $G$ is $\sigma$-primary, that is, $H/K$ is a $\sigma_{i}$-group for some $i$.
But since $W_{i}$ is  nilpotent, $H/K$ is a elementary abelian group.
It follows that $G$ is soluble. Hence we only need to prove that $G$ is $\sigma$-soluble.
Suppose that the assertion is false,
and let $G$ be a counterexample of minimal order.
Then clearly $t > 1$, and $p=2 \in \pi(W_{1})$ by the well-known Feit-Thompson's theorem.
Without loss of generality, we can assume that $W_{i}$ is a $\sigma_{i}$-group
for all $i=1, 2,\cdots, t$.

$(1)$ $O_{\sigma_{1}}(G)=1$.

Assume that $N=O_{\sigma_{1}}(G)\neq1$.
Note that if $W_{1}=N$, then $G/N$ is a $\sigma_{1}'$-group, so $G/N$ is soluble by the well-known Feit-Thompson's theorem and so $G$ is $\sigma$-soluble.
We may, therefore, assume that $W_{1} \neq N$.
Then $W_{1}/N$ is a non-identity Hall $\sigma_{1}$-subgroup of $G/N$.
Let $M/N$ be a maximal subgroup of $W_{1}/N$.
Then $M$ is a maximal subgroup of $W_{1}$.
By the hypothesis and Lemma \ref{main} (2),
$M/N$ is weakly $\sigma$-permutable in $G/N$.
The minimal choice of $G$ implies that $G/N$ is $\sigma$-soluble.
Consequently, $G$ is $\sigma$-soluble.
This contradiction shows that (1) holds.

$(2)$ $O_{\sigma_{1}'}(G)=1$.

Assume that $R=O_{\sigma_{1}'}(G)\neq1$.
Then $W_{1}R/R$ is a Hall $\sigma_{1}$-subgroup of $G/R$.
Let $M/R$ be a maximal subgroup of $W_{1}R/R$.
Then $M=(M\cap W_{1})R$.
Since $W_{1}$ is nilpotent and $|W_{1}R/R:M/R|=|W_{1}R/R:(M\cap W_{1})R/R|=|W_{1}:M\cap W_{1}|$,
$M\cap W_{1}$ is a maximal subgroup of $W_{1}$.
By the hypothesis and Lemma \ref{main} (3),
$M/R=(M\cap W_{1})R/R$ is weakly $\sigma$-permutable in $G/R$.
This shows that $G/R$ satisfies the hypothesis.
The choice of $G$ implies that $G/R$ is $\sigma$-soluble.
By the well-known Feit-Thompson's theorem,
we know that $R$ is soluble.
It follows that $G$ is $\sigma$-soluble, a contradiction.

$(3)$ If $R\neq1$ is a minimal normal subgroup of $G$,
then $R$ is not $\sigma$-soluble and $G=RW_{1}$.

If $R$ is $\sigma$-soluble,
then $R$ is a $\sigma_{i}$-group for some $\sigma_{i}\in \sigma(G)$.
So $R\leq O_{\sigma_{1}}(G)$ or $R\leq O_{\sigma_{1}^{'}}(G)$, a contradiction.
Therefore, $R$ is not $\sigma$-soluble.
Assume that $RW_{1}<G$.
Then by the hypothesis and Lemma \ref{main} (1),
$RW_{1}$ satisfies the hypothesis.
Hence $RW_{1}$ is $\sigma$-soluble by the choice of $G$.
It follows from Lemma \ref{soluble} that $R$ is $\sigma$-soluble.
This contradiction shows that $G=RW_{1}$.

$(4)$ $G$ has a unique minimal normal subgroup $R$.

By (3),
$G=RW_{1}$ for every non-identity minimal normal subgroup $R$ of $G$.
Then clearly, $G/R$ is $\sigma$-soluble.
Hence by Lemma \ref{soluble} $G$ has a unique minimal normal subgroup,
which is denoted by $R$.

$(5)$ $W_{1}$ is a $2$-group.

Let $q\in \pi(W_{1})\backslash\{2\}$.
As $W_{1}$ is nilpotent,
there exists two maximal subgroups $M_{1}$ and $M_{2}$ of $W_{1}$
such that $|W_{1}:M_{1}|=q$ and $|W_{1}:M_{2}|=2$.
By the hypothesis,
there exists $\sigma$-subnormal subgroups $T_{i}$ of $G$,
such that $G=M_{i}T_{i}$ and $M_{i}\cap T_{i}\leq (M_{i})_{\sigma G}$,
$i=1,2$.
By Lemma \ref{permutable} (4),  $(M_{i})_{\sigma G}$ is $\sigma$-subnormal in $G$.
Then by Lemma \ref{subnormal} (8),
$(M_{i})_{\sigma G}\leq O_{\sigma_{1}}(G)=1$, $i=1,2$.
Hence $M_{i}\cap T_{i}=1,i=1,2$.
Consequently, $|G:T_{i}|=|M_{i}:M_{i}\cap T_{i}|=|M_{i}|$, $i=1,2$,
which implies that $|G:T_{i}|$ is a $\sigma_{1}$-number
for $i=1, 2$.
Hence $O^{\sigma_{1}}(T_{i})=O^{\sigma_{1}}(G)$ for $i=1,2$ by Lemma \ref{subnormal} (7).
Since $t > 1$, $O^{\sigma_{1}}(G)>1$. It follows that $1\neq O^{\sigma_{1}}(G)\leq (T_{i})_{G}$ for $i=1,2$.
Then by (4), $R\leq (T_{1})_{G} \cap (T_{2})_{G}\leq T_{1} \cap T_{2}$.
It is clear that $W_{1}\cap R$ is a Hall $\sigma_{1}$-subgroup of $R$,
and $W_{1}\cap R\neq1$ by (2).
So $1\neq W_{1}\cap R\leq T_{1}\cap T_{2}\cap W_{1}$.
Since $G=M_{1}T_{1}=W_{1}T_{1}=M_{2}T_{2}=W_{1}T_{2}$,
where $M_{1}\cap T_{1}=1$ and $M_{2}\cap T_{2}=1$,
we have that $|W_{1}\cap T_{1}|=|W_{1}: M_{1}|=q$ and
$|W_{1}\cap T_{2}|=|W_{1}: M_{2}|=2$.
Therefore $(W_{1}\cap T_{1})\cap (W_{1}\cap T_{2})=1$,
which implies that $1\neq W_{1}\cap R\leq T_{1}\cap T_{2}\cap W_{1}=(T_{1}\cap W_{1})\cap (T_{2}\cap W_{1})=1$.
This contradiction shows that $W_{1}$ is a $2$-group.

$(6)$ Final contradiction.

Let $P_{1}$ be a maximal subgroup of $W_{1}$.
Then $|W_{1}:P_{1}|=2$.
By the hypothesis,
there exists a $\sigma$-subnormal subgroup $K$ of $G$
such that $G=P_{1}K$ and $P_{1}\cap K\leq (P_{1})_{\sigma G}$.
By (1) and Lemma \ref{subnormal} (8), $(P_{1})_{\sigma G}=1$,
Hence $|K|_{2}=2$,
and so $K$ is $2$-nilpotent by \cite[IV, Theorem 2.8]{HU}.
Let $K_{2^{'}}$ be the normal Hall $2^{'}$-subgroup of $K$.
Then $1 \neq K_{2^{'}}$ is $\sigma$-subnormal in $G$,
and so $K_{2^{'}}\leq O_{\sigma_{1}^{'}}(G)=1$ by Lemma \ref{subnormal}(8).
The finial contradiction completes the proof.
\end{proof}

{\bf Proof of Theorem \ref{Th3}.}
Assume that the assertion is false,
and let $G$ be a counterexample of minimal order.

$(1)$ $G$ is soluble.

Let $q$ is the smallest prime dividing $|G|$. Without loss of generality, we may assume that $q\in \pi(W_{1})$.
If $W_{1}$ is cyclic,
then the Sylow $q$-subgroup of $G$ is cyclic.
Hence $G$ is $q$-nilpotent by \cite[IV, Theorem 2.8]{HU}
and so $G$ is soluble.
If $W_{1}$ is non-cyclic,
then by Proposition \ref{prop1}, $G$ is soluble.
Hence we always have that $G$ is soluble.

$(2)$ The hypothesis holds on $G/R$ for every non-identity minimal normal subgroup $R$ of $G$. Consequently $G/R$ is supersoluble.

It is clear that ${\overline\mathcal {H}}=\{1,W_{1}R/R, W_{2}R/R, \cdots, W_{t}R/R\}$ is a complete Hall $\sigma$-set of $G/R$ and $W_{i}R/R\simeq W_{i}/W_{i}\cap R$ is nilpotent.
By (1), $R$ is an elementary abelian $p$-group for some prime $p$.
Without loss of generality, we can assume that $R\leq W_{1}$.
If $W_{1}/R$ is non-cyclic,
then $W_{1}$ is non-cyclic.
For every maximal subgroup $M/R$ of $W_{1}/R$,
we have that $M$ is a maximal subgroup of $W_{1}$.
Then by the hypothesis and Lemma \ref{main} (2),
$M/R$ is weakly $\sigma$-permutable in $G/R$.
Now assume that $W_{i}R/R$ is non-cyclic for $i\neq1$ and $M/R$ be a maximal subgroup of $W_{i}R/R$.
Then $M=(M\cap W_{i})R$.
As $W_{i}$ is nilpotent, $|W_{i}R/R:M/R|=|W_{i}R/R:(M\cap W_{i})R/R|=|W_{i}:M\cap W_{i}|$ is a prime.
Hence $M\cap W_{i}$ is a maximal subgroup of $W_{i}$.
By the hypothesis and Lemma \ref{main} (3),
$M/R=(M\cap W_{i})R/R$ is weakly $\sigma$-permutable in $G/R$.
This shows that the hypothesis holds for $G/R$.
Hence $G/R$ is supersoluble by the choice of $G$.

$(3)$ $R$ is the unique minimal normal subgroup of $G$,
$\Phi(G)=1$, $C_{G}(R)=R=F(G)=O_{p}(G)$ and $|R|>p$ for some prime $p$ (It follows from (2)).

$(4)$ For some $i \in \{1,2,\cdots,t \}$, $W_{i}$ is a $p$-group. Without loss of generality, we may assume that $W_{1}$ is a $p$-group.

Since $R$ is a $p$-group, $R\leq W_{i}$ for some $i \in \{1,2,\cdots,t \}$.
Moreover, since $C_{G}(R)=R$ and $W_{i}$ is a nilpotent group,
we have that $W_{i}$ is a $p$-group.

$(5)$ Final contradiction.

Since $ \Phi(G) =1$,  $R\nleqslant \Phi(W_{1})$ \cite[Chapter III, Lemma 3.3]{HU}.
Hence there exists a maximal subgroup $V$ of $W_{1}$
such that $W_{1}=RV$.
Let $E=R\cap V$.
Then $|R:E|=|RV:V|=|W_{1}:V|=p$.
Hence $E$ is a maximal subgroup of $R$ and $1\neq E\unlhd W_{1}$.
Since $|R|>p$ and $R\leq W_{1}$,
$W_{1}$ is non-cyclic.
Hence by the hypothesis,
there exists a $\sigma$-subnormal subgroup $T$ of $G$
such that $G=VT$ and $V\cap T\leq V_{\sigma G}$.
Since $|G:T|$ is $p$-number, $O^{p}(T) = O^{\sigma_{1}}(T)=O^{\sigma_{1}}(G)$ by Lemma \ref{subnormal} (7).
So $|G:T_{G}|$ is $p$-number.
It follows that $T_{G}\neq1$ and $R\leq T_{G}\leq T$ by (2).
Since $V_{\sigma G}$ is $\sigma$-subnormal in $G$ by Lemma \ref{permutable} (4),
we have that $V_{\sigma G}\leq O_{\sigma_{1}}(G)=O_{p}(G)=R$ by Lemma \ref{subnormal} (8).
Hence $E=R\cap V\leq T\cap V\leq V_{\sigma G}\leq R$.
But since $E$ is a maximal subgroup of $R$,
it follows that $V_{\sigma G}=R$ or $V_{\sigma G}=E$.
In the former case, we have that $R\leq V$, a contradiction.
In the later case, $E=V_{\sigma G}$ is $\sigma$-permutable in $G$ by Lemma \ref{permutable} (4)
and $E$ is a $\sigma_{1}$-group.
It follows from Lemma \ref{normalizes} that $O^{\sigma_{1}}(G)\leq N_{G}(E)$.
Hence $E\unlhd G$,
which contradicts the minimality of $R$.
The final contradiction completes the proof of the theorem.

 \section{Proof of Theorem \ref{Th4}}

Assume that the assertion is false and let $(G,E)$ be a counterexample with $|G|+|E|$ minimal.
Without loss of generality, we can assume that $W_{i}$ is a $\sigma_{i}$-group
for all $i=1, 2,\cdots, t$.
We now proceed with the proof via the following steps.

$(1)$ $E$ is supersoluble.

In fact, $\{1,W_{1}\cap E, W_{2}\cap E,\cdots, W_{t}\cap E\}$ is a complete Hall $\sigma$-set of $E$ and $W_{i}\cap E$ is nilpotent.
Consequently, $E$ is a $\sigma$-full group of Sylow type.
By Lemma \ref{main} (1) and Theorem \ref{Th3}.
Hence $E$ is supersoluble.

$(2)$ If $R$ is a minimal normal subgroup of $G$ contained in $E$,
then $R$ is a $p$-group for some prime $p$ and the hypothesis holds for $(G/R,E/R)$.
Therefore $E/R\leq Z_{\mathfrak{U}}(G/R)$.

By (1), $R$ is a $p$-group for some $p$.
Without loss of generality, we can assume that $R\leq W_{1}\cap E$.
It is clear that ${\overline\mathcal {H}}=\{1,W_{1}/R, W_{2}R/R, \cdots, W_{t}R/R\}$
is a complete Hall $\sigma$-set of $G/R$ and $W_{i}R/R\simeq W_{i}/W_{i}\cap R$ is nilpotent.
Let $M/R$ be a maximal subgroup of $(W_{1}\cap E)/R$.
Then by the hypothesis and Lemma \ref{main} (2), $M/R$ is weakly $\sigma$-permutable in $G/R$.
Now let $V/R$ be a maximal subgroup of $W_{i}R/R\cap E/R=(W_{i}\cap E)R/R$, $i=2,\cdots, t$.
Then $V=(V\cap W_{i})R$.
Since $W_{i}R/R\cap E/R$ is nilpotent, $|W_{i}\cap E:V\cap W_{i}|=|W_{i}R\cap E:(V\cap W_{i})R|=|W_{i}R/R\cap E/R:V/R|$ is a prime, so $V\cap W_{i}$ is a maximal subgroup of $W_{i}\cap E$.
Then by the hypothesis and Lemma \ref{main} (3),
$V/R=(V\cap W_{i})R/R$ is weakly $\sigma$-permutable in $G/R$, $i=2,\cdots, t$.
This shows that $(G/R,E/R)$ satisfies the hypothesis.
Thus $E/R\leq Z_{\mathfrak{U}}(G/R)$ by the choice of $(G,E)$.

$(3)$ $R$ is the unique minimal normal subgroup of $G$ contained in $E$,
$|R|>p$ and $O_{p'}(E)=1$.

Let $L$ be a minimal normal subgroup of $G$ contained in $E$ such that $R\neq L$.
Then $E/R\leq Z_{\mathfrak{U}}(G/R)$ and $E/L\leq Z_{\mathfrak{U}}(G/L)$ by (2),
and clearly, $|R|>p$.
It follows that $LR/L\leq Z_{\mathfrak{U}}(G/L)$,
so $|R|=p$ by the $G$-isomorphism $RL/L\simeq R$,
a contradiction.
Hence $R$ is the unique minimal normal subgroup of $G$ contained in $E$. Consequently, $O_{p'}(E)=1$.
Hence (3) holds.

Without loss of generality, we may assume $p \in \pi(W_{1})$.

$(4)$ $E$ is a $p$-group and so $E\cap W_{1}=E$ and $E\cap W_{i}=1$ for $i= 2,3,\cdots,t$.

Let $q$ be the largest prime dividing $|E|$ and let $Q$ be a Sylow $q$-subgroup of $E$.
Since $E$ is supersoluble by (1) (see \cite[Chapter \uppercase\expandafter{\romannumeral6}, Theorem 9.1] {HU}),
$Q$ is characteristic in $E$.
Then $Q$ is normal in $G$.
Hence by (3) we have that $q=p$ and $F(E)=Q=O_{p}(E)=P$ is a Sylow $p$-subgroup of $E$.
Thus $C_{E}(P)\leq P$ (see \cite[Theorem 1.8.18]{W1}).
But since $P\leq W_{1}\cap E$ and $W_{1}\cap E$ is nilpotent,
we have that $P= W_{1}\cap E$.
Since $P\cap W_{1}=P= W_{1}\cap E$ and $P\cap W_{i}=1$ for all $i=2,\cdots, t$, the hypothesis holds for $(G,P)$.
If $P<E$,
then $R\leq P\leq Z_{\mathfrak{U}}(G)$ by the choice of $(G,E)$.
It follows that $|R|=p$, a contradiction.
Hence $E=P$ is a $p$-group,
and so $E\leq W_{1}$.

$(5)$ $\Phi(E)=1$, so $E$ is elementary abelian $p$-group.

Assume that $\Phi(E)\neq 1$.
Then clearly, $(G/\Phi(E), E/\Phi(E))$ satisfies the hypothesis.
Hence $E/\Phi(E) \leq  Z_{\mathfrak{U}}(G/\Phi(E))$.
It follows from (4) and Lemma \ref{XC} that $E \leq  Z_{\mathfrak{U}}(G)$, a contradiction.
Thus we have (5).

$(6)$ Final contradiction.

Let $R_{1}$ be a maximal subgroup of $R$ such that $R_{1}\unlhd W_{1}$.
Then $|R_{1}| > 1$ by (3).
Claim (5) implies that $R$ has a complement $S$ in $E$.
Let $V=R_{1}S$.
Then $R\cap V=R_{1}$ and $V$ is a maximal subgroup of $E$.
Hence by (4) and the hypothesis,
there exists a $\sigma$-subnormal subgroup $T$ of $G$
such that $G=VT$ and $V\cap T\leq V_{\sigma G}$.
Then $G=VT=ET$ and $E=V(E\cap T)$.
By (5), it is easy to see that $1\neq E\cap T\unlhd G$.
Hence $R\leq E\cap T$ by (3), and so $R_{1}=R\cap V\leq E\cap T\cap V=V\cap T\leq V_{\sigma G}$.
Consequently, $R_{1}\leq V_{\sigma G}\cap R\leq R$.
It follows that $R=V_{\sigma G}\cap R$ or $R_{1}=V_{\sigma G}\cap R$.
In the former case,
$R\leq V$,  which contradicts the fact that $R_{1}=R\cap V$.
Thus $R_{1}=V_{\sigma G}\cap R$.
By Lemma \ref{permutable}(4),
we have that $V_{\sigma G}$ is $\sigma$-permutable in $G$, so $O^{\sigma_{1}}(G)\leq N_{G}(V_{\sigma G})$ by Lemma \ref{normalizes}.
Hence $O^{\sigma_{1}}(G)\leq N_{G}(V_{\sigma G}\cap R)=N_{G}(R_{1})$.
Moreover, since $R_{1}\unlhd W_{1}$, we obtain that $R_{1}\unlhd G$.
This implies that $R_{1}=1$.
The final contradiction completes the proof.

\end{document}